\documentclass[12pt, reqno]{amsart}
\usepackage{amsmath, amsthm, amscd, amsfonts, amssymb, mathtools, color, hyperref}
\usepackage[dvipsnames]{xcolor}

\newtheorem{theorem}{Theorem}[section]
\newtheorem{lemma}[theorem]{Lemma}

\newtheorem{corollary}[theorem]{Corollary}

\theoremstyle{definition}

\theoremstyle{remark}
\newtheorem{remark}[theorem]{Remark}
\numberwithin{equation}{section}

\begin{document}
\setcounter{page}{1}

\title[Eigenvalue value location]
{Eigenvalue location of certain matrix polynomials}

\author[Pallavi, Shrinath and Sachindranath]{Pallavi. B, Shrinath Hadimani 
and Sachindranath Jayaraman}
\address{School of Mathematics\\ 
Indian Institute of Science Education and Research Thiruvananthapuram\\ 
Maruthamala P.O., Vithura, Thiruvananthapuram -- 695 551, Kerala, India.}
\email{(pallavipoorna20, srinathsh3320, sachindranathj)@iisertvm.ac.in, 
sachindranathj@gmail.com}

	

\begin{abstract} 
It is known that a matrix polynomial with unitary matrix coefficients has its eigenvalues 
in the annular region $\frac{1}{2} < |\lambda| < 2$. We prove in this short note that 
under certain assumptions, matrix polynomials with either doubly stochastic matrix 
coefficients or Schur stable matrix coefficients also have eigenvalues in similar annular regions. 
\end{abstract}
	
\maketitle
	
\section{Introduction}\label{sec-1}

We work over the field $\mathbb{C}$ of complex numbers. $M_n(\mathbb{C})$ denotes the space 
of $n \times n$ matrices over $\mathbb{C}$. The notations $||\cdot||_2$ and 
$||\cdot||_\infty$ will denote respectively the spectral norm and the maximum row sum norm. 
A matrix polynomial of size $n$ and degree $m$ is a function $P$ from $\mathbb{C}$ to 
$M_n(\mathbb{C})$, given by $P(\lambda) = \displaystyle \sum_{i=0}^{m} A_i\lambda^i$, 
where $A_i \in M_n(\mathbb{C})$ and $A_m \neq 0$. A scalar $\lambda_0 \in \mathbb{C}$ is 
called an eigenvalue of $P(\lambda)$ if $P(\lambda_0)v = 0$ for some nonzero vector 
$v \in \mathbb{C}^n$. The nonzero vector $v \in \mathbb{C}^n$ satisfying the 
equation $P(\lambda_0)v=0$ is called an eigenvector of $P(\lambda)$ corresponding to an 
eigenvalue $\lambda_0$. Equivalently, for a given matrix polynomial $P(\lambda), 
\lambda_0 \in \mathbb{C}$ is an eigenvalue, if $\text{det}P(\lambda_0) = 0$. Notice that 
for a matrix polynomial of size $n$ and degree $m$, there are at most $mn$ number of
eigenvalues. Moreover, $\lambda_0 = 0$ is an eigenvalue of $P(\lambda)$ if and only 
if $A_0$ is singular. We say $\infty$ is an eigenvalue of $P(\lambda)$, if the reverse matrix 
polynomial $\widehat{P} (\lambda):= \lambda^m P(\frac{1}{\lambda}) = A_0 \lambda^m + 
A_1 \lambda^{m-1} + \cdots + A_{m-1} \lambda + A_m$ has zero as an eigenvalue. Therefore if 
the leading coefficient is nonsingular, $P(\lambda)$ has exactly $mn$ number of finite eigenvalues, 
counting multiplicities. If the leading coefficient of $P(\lambda)$ is a nonsingular matrix then 
define a monic matrix polynomial corresponding to $P(\lambda)$ as, $P_U(\lambda):= I \lambda^m + 
U_{m-1} \lambda^{m-1} + \cdots + U_1 \lambda + U_0$, where $U_i= A_m^{-1}A_i$ for $i= 0,1,\dots, m-1$ 
and a block companion matrix corresponding to $P(\lambda)$ as 
$C:= \begin{bmatrix}
0 & I & 0 & \cdots &  0\\
0 & 0 & I & \cdots & 0 \\
\vdots & \vdots &  \vdots & \ddots & \vdots\\
0 & 0 & 0 & \cdots & I\\
-U_0 & -U_1 & -U_2 & \cdots & -U_{m-1}
\end{bmatrix}$. The eigenvalues of $P(\lambda)$, $P_U(\lambda)$ and $C$ are the same 
(see \cite{Higham-Tisseur} for details). 

\medskip
Eigenvalue location of matrix polynomials is 
a challenging problem, with interesting applications in applied mathematics. For instance, 
information about the location of eigenvalues is helpful in determining pseudospectra of 
matrix polynomials \cite{Tisseur-Higham}. The papers \cite{Betcke-Higham-Tisseur} 
and \cite{Tisseur-Meerbergen} are a good source of examples. Several interesting methods to 
bound the eigenvalues of matrix polynomials exist in the literature (see for instance 
\cite{Higham-Tisseur, Bini-Noferini-Sharify} and \cite{ Le-Du-Nguyen}). This work aims to 
provide eigenvalue bounds for matrix polynomials whose coefficients are either doubly stochastic 
matrices or Schur stable matrices. We end this section with a lemma due to Cauchy that will be used 
later.

\begin{lemma}[\cite{Cauchy}] \label{Lem-Cauchy bound for roots of polynomials}
Let $p(\lambda) = a_n\lambda^n + \dots + a_1\lambda + a_0$ be a polynomial of degree $n$. 
If $\lambda_0 \in \mathbb{C}$ is a root of $p(\lambda)$, then $|\lambda_0| \leq 1 + 
\max \Big\{\frac{|a_{0}|}{|a_n|},\frac{|a_1|}{|a_n|},\dots,\frac{|a_{n-1}|}
{|a_n|}\Big\}$.
\end{lemma}

\section{Main Results}\label{sec-2}

The main results are presented in this section. For ease of reading, this is subdivided 
into subsections, which are self-explanatory.
 
\medskip

\subsection{Location of eigenvalues of matrix polynomials with doubly stochastic matrix
coefficients}\label{sec-2.2}\hspace*{\fill} 

\medskip

We consider matrix polynomials whose coefficients are doubly stochastic matrices in this 
section. A nonnegative square matrix is doubly stochastic if all the row and column 
sums are $1$. A recent result due to Cameron (Theorem $3.2$, \cite{Cameron}) says that if 
$\lambda_0$ is an eigenvalue of a matrix polynomial with unitary matrix coefficients, then 
$\frac{1}{2} < |\lambda_0| < 2$. Doubly stochastic matrices being convex combination  
of permutation matrices, a natural question is to ask if the eigenvalues of matrix 
polynomials with doubly stochastic matrix coefficients lie in some annular region. It turns 
out that the eigenvalues lie in the same annular region as in the unitary case. Since $||A||_2=1$ 
for any doubly stochastic matrix $A$, the proof is very much the same as 
in the unitary case. We therefore state the theorem without proof. We introduce some notations 
before stating the theorem. Let $\mathcal{D} = \big\{P(\lambda) = A_m\lambda^m + A_{m-1}\lambda^{m-1} 
+ \cdots + A_0: A_i$ are $n\times n$ doubly stochastic matrices and 
$A_m, A_0$ are $n \times n$ permutation matrices, $n, m \in \mathbb{N}$ $\big\}$ and let 
$\sigma(\mathcal{D})=\big\{|\lambda_0|:\lambda_0$ is an eigenvalue of 
$P(\lambda)\in \mathcal{D}\big\}$.

\begin{theorem}\label{Thm-1-eigenvalue-doubly stochastic-coefficients}
Let $P(\lambda)=A_m\lambda^m+A_{m-1}\lambda^{m-1}+\cdots+A_0 \in \mathcal{D}$. If 
$\lambda_0 \in \mathbb{C}$ is an eigenvalue of $P(\lambda)$ then 
$\frac{1}{2}<|\lambda_0|<2$.
\end{theorem}

\medskip

The bounds given in Theorem \ref{Thm-1-eigenvalue-doubly stochastic-coefficients} are 
optimal, as we prove below.

\begin{theorem}\label{Thm-2-eigenvalue-doubly stochastic-coefficients}
$\inf (\sigma(\mathcal{D}))=\frac{1}{2}$ and $\sup (\sigma(\mathcal{D}))=2$.
\end{theorem}
\begin{proof}
From Theorem \ref{Thm-1-eigenvalue-doubly stochastic-coefficients}, 
$\frac{1}{2}$ is a lower bound for $\sigma(\mathcal{D})$. We show that any number greater 
than $\frac{1}{2}$ cannot be a lower bound for $\sigma(\mathcal{D})$. If 
$\frac{1}{2}<r<1$, then $\displaystyle \sum_{i=1}^{\infty}r^i >1$. Therefore there 
exists $d \in \mathbb{N}$, such that $\displaystyle \sum_{i=1}^{d}r^i >1$. Consider, 
$q(\lambda):=\lambda^d+\lambda^{d-1}+\cdots+\lambda-1$. Then $q(r) = \displaystyle 
\sum_{i=1}^{d}r^i-1>0$, and $q(\frac{1}{2}) = \displaystyle 
\sum_{i=1}^{d} \Big(\frac{1}{2}\Big)^i-1<0$. The intermediate value theorem 
guarantees that $q(\lambda)$ has a root in the interval $(\frac{1}{2},r)$. Now consider 
the matrix polynomial, $P(\lambda):= 
I \lambda^d +I\lambda^{d-1}+\cdots+I\lambda +I'\in \mathcal{D}$, where $I=\begin{bmatrix}
1 & 0\\
0 & 1
\end{bmatrix}$ and $I^{\prime}=\begin{bmatrix}
0 & 1\\
1 & 0
\end{bmatrix}$. Then we have $P(\lambda)=\begin{bmatrix}
\lambda^d+\lambda^{d-1}+\cdots+\lambda & 1 \\
1 & \lambda^d+\lambda^{d-1}+\cdots+\lambda
\end{bmatrix}$. Therefore $\text{det}P(\lambda)=(\lambda^d+\lambda^{d-1}+\cdots+
\lambda)^2-1=(\lambda^d+\lambda^{d-1}+\cdots+\lambda+1)(\lambda^d+\lambda^{d-1}+\cdots+
\lambda-1)$. Thus $q(\lambda)$ is a factor of $\text{det}P(\lambda)$. Since $q(\lambda)$ 
has a root in the interval $(\frac{1}{2},r)$, $P(\lambda)$ has an eigenvalue whose 
modulus is less than $r$. Thus, $r$ is not a lower bound for $\sigma(\mathcal{D})$, 
thereby proving that $\inf (\sigma(\mathcal{D}))=\frac{1}{2}$.  
Again by Theorem \ref{Thm-1-eigenvalue-doubly stochastic-coefficients}, $2$ is an 
upper bound for $\sigma(\mathcal{D})$. We now show that any number less than $2$ cannot 
be an upper bound for $\sigma(\mathcal{D})$. Let $1 < R < 2$. Define $w_m(\lambda)=\lambda^m-
\lambda^{m-1}-\cdots-\lambda-1$, for $m\in \mathbb{N}$. Suppose $R$ is an upper bound for 
$\sigma(\mathcal{D})$, then $R^m\geq R^{m-1}+\cdots+R+1$. Otherwise, $R^m<R^{m-1}+\cdots+
R+1$. This implies $w_m(R)=R^m-R^{m-1}-\cdots-R-1 <0$. We also have $w_m(2)>0$. Once again, 
by the intermediate value theorem $w_m(\lambda)$ has a root in the 
interval $(R,2)$. Now consider, $Q(\lambda):=I\lambda^m + I^{\prime} \lambda^{m-1} + 
\cdots + I^{\prime} \lambda + I^{\prime} \in \mathcal{D}$. Then we have, $Q(\lambda) = 
\begin{bmatrix}
\lambda^m &\lambda^{m-1}+\cdots+\lambda+1 \\
\lambda^{m-1}+\cdots+\lambda+1 & \lambda^m
\end{bmatrix}$. Therefore $\text{det}Q(\lambda)=(\lambda^m)^2-(\lambda^{m-1}+\cdots+
\lambda+1)^2=(\lambda^m-\lambda^{m-1}-\cdots-\lambda-1)(\lambda^m+\lambda^{m-1}+\cdots
+\lambda+1)$. Thus $w_m(\lambda)$ is a factor of $\text{det}Q(\lambda)$. Since 
$w_m(\lambda)$ has a root in the interval $(R,2)$, $Q(\lambda)$ has an eigenvalue 
whose modulus is greater than $R$. This is a contradiction to the assumption that 
$R$ is an upper bound for $\sigma(\mathcal{D})$. Hence $R^m \geq R^{m-1}+\cdots+R+1 = 
\dfrac{1-R^m}{1-R}$. Since $ 1-R<0$ we have, $(1-R)R^m \leq 1-R^m$, which implies 
$2-R\leq \dfrac{1}{R^m}$. Since this is true for all $m\in \mathbb{N}$, as 
$m \rightarrow \infty$ we get, $2-R\leq 0$; that is, $2\leq R$, a contradiction to 
the choice of $R$. Thus $R$ is not an upper bound for $\sigma(\mathcal{D})$. Hence 
$\sup (\sigma(\mathcal{D}))=2$.
\end{proof}

\medskip
\begin{remark}\label{Rem-doubl stochastic}
If the leading coefficient or the constant term is a doubly stochastic matrix, 
but not a permutation matrix, then the eigenvalues may not necessarily lie in the 
region $\frac{1}{2} < |\lambda| < 2$. We refer to Remark $2.9$ of \cite{PB-SH-SJ-1} 
for details. We also wish to remark that both Theorems 
\ref{Thm-1-eigenvalue-doubly stochastic-coefficients} and 
\ref{Thm-2-eigenvalue-doubly stochastic-coefficients} were stated in \cite{PB-SH-SJ-1} 
without proof.
\end{remark}

\medskip
\noindent
We now prove that any matrix polynomial of degree at least $2$ from the collection 
$\mathcal{D}$ has two distinct eigenvalues. Note that the eigenvalues of a doubly stochastic matrix 
are of absolute value less than or equal to $1$.
    
\begin{corollary}\label{Cor-3-doubly-stochastic}
Let $P(\lambda) = A_m\lambda^m + A_{m-1}\lambda^{m-1} + \cdots + A_1\lambda + A_0$ be 
an $n\times n$ matrix polynomial where the $A_i$'s are $n\times n$ doubly stochastic 
matrices and $A_m, A_0$ are permutation matrices with $m \geq 2$. Then $P(\lambda)$ 
has at least two distinct eigenvalues.
\end{corollary}
    
\begin{proof}
We prove this by contradiction.	 Suppose all the eigenvalues of $P(\lambda)$ are same, 
say $\lambda_0$. The monic matrix polynomial corresponding to $P(\lambda)$ is given by 
$P_U(\lambda) = I \lambda^m + B_{m-1} \lambda^{m-1} + \cdots+B_1 \lambda + B_0$ where, 
$B_i = A_m^{-1}A_i$ for $1 \leq i \leq m-1$. Since $A_m$ is a permutation matrix, the 
matrices $B_i$'s are doubly stochastic matrices and $B_0$ is a permutation matrix. 
Consider the corresponding block companion matrix of $P(\lambda)$,
$C=\begin{bmatrix}
0 & I & 0 & \cdots & 0\\
0 & 0 & I & \cdots & 0\\
\vdots & \vdots & \vdots & \ddots & \vdots\\
0 & 0 & 0 & \cdots & I\\
-B_0 & -B_1 & -B_2  & \cdots & -B_{m-1}  
\end{bmatrix}$. 
Then we have, $\text{trace}(C) = - \text{trace}(B_{m-1})$. Let $\mu_1, \ldots,\mu_n$ 
be the eigenvalues of $B_{m-1}$. We then have  
$\displaystyle \sum_{i=1}^{mn} \lambda_0 = -\sum_{i=1}^{n} \mu_i$. 
This implies $mn \lambda_0 = \displaystyle -\sum_{i=1}^{n}\mu_i$. On taking modulus 
we have $mn |\lambda_0| = \displaystyle \Big|\sum_{i=1}^{n}\mu_i\Big| \leq 
\displaystyle \sum_{i=1}^{n}|\mu_i| \leq \displaystyle \sum_{i=1}^{n}1 = n$.
Thus $|\lambda_0| \leq \frac{1}{m} \leq \frac{1}{2}$, a contradiction to Theorem 
\ref{Thm-1-eigenvalue-doubly stochastic-coefficients}. Therefore $P(\lambda)$ has at 
least two distinct eigenvalues.
\end{proof}
 
From the above proof one may observe that Corollary \ref{Cor-3-doubly-stochastic} is 
also true when all the coefficients are unitary matrices. We prove below that 
an $m^{th}$ degree matrix polynomial with doubly stochastic matrix coefficients has at least $m$ 
distinct eigenvalues on the unit circle.

\begin{theorem}\label{Thm-4-doubly-stochastic}
Let $P(\lambda) = A_m \lambda^m +\dots + A_1 \lambda + A_0$, where $A_i$'s are 
$n \times n$ doubly stochastic matrices. Then $P(\lambda)$ has at least $m$ distinct
eigenvalues on the unit circle.
\end{theorem}

\begin{proof}
Since the coefficients are doubly stochastic matrices, we can write,
\begin{center}
$A_k =
\begin{bmatrix}
a^{(k)}_{11} & a^{(k)}_{12} & \ldots & a^{(k)}_{1(n-1)} &
1-\displaystyle \sum_{j=1}^{n-1}a^{(k)}_{1j}\\
a^{(k)}_{21} & a^{(k)}_{22} & \ldots & a^{(k)}_{2(n-1)} &
1-\displaystyle \sum_{j=1}^{n-1}a^{(k)}_{2j}\\
\vdots & \vdots & \ddots & \vdots & \vdots \\
a^{(k)}_{(n-1)1} & a^{(k)}_{(n-1)2} & \ldots & a^{(k)}_{(n-1)(n-1)} &
1 - \displaystyle \sum_{j=1}^{n-1}a^{(k)}_{(n-1)j}\\
1 - \displaystyle \sum_{i=1}^{n-1}a^{(k)}_{i1} & 1 - \displaystyle \sum_{i=1}^{n-1}
a^{(k)}_{i2} & \ldots & 1 - \displaystyle \sum_{i=1}^{n-1}a^{(k)}_{i(n-1)} &
\displaystyle \sum_{i,j=1}^{n-1} a^{(k)}_{ij} - (n-2)
\end{bmatrix}$
\end{center}
for $k = 0, 1, \dots , m$ and $0 \leq a^{(k)}_{ij} \leq 1$ for $i,j = 1,2,\dots,n-1$.
Therefore, $P(\lambda) = $
\begin{center}
$\begin{bmatrix}
\displaystyle \sum_{k=0}^{m}a_{11}^{(k)}\lambda^k & \cdots & \displaystyle 
\sum_{k=0}^{m}a_{1(n-1)}^{(k)}\lambda^k & \displaystyle \sum_{k=0}^{m}\Big(1 
-\sum_{j=1}^{n-1}a^{(k)}_{1j}\Big) \lambda^k \\
\vdots & \ddots & \vdots & \vdots \\
\displaystyle \sum_{k=0}^{m}a_{(n-1)1}^{(k)}\lambda^k & \cdots & \displaystyle 
\sum_{k=0}^{m}a_{(n-1)(n-1)}^{(k)}\lambda^k & \displaystyle \sum_{k=0}^{m}\Big(1 - 
\sum_{j=1}^{n-1}a^{(k)}_{(n-1)j}\Big) \lambda^k \\
\displaystyle \sum_{k=0}^{m}\Big(1 - \sum_{i=1}^{n-1}a^{(k)}_{i1}\Big) \lambda^k & \cdots 
& \displaystyle \sum_{k=0}^{m}\Big(1 -  \sum_{i=1}^{n-1}a^{(k)}_{i(n-1)}\Big) \lambda^k & 
\displaystyle \sum_{k=0}^{m}\Big(\sum_{i,j=1}^{n-1}a^{(k)}_{ij}-n+2\Big) \lambda^k
\end{bmatrix}$.
\end{center}
Now by performing the following elementary row operations we get a sequence of matrix
polynomials 
\begin{center}
$P(\lambda) \hspace{0.2cm} \tiny{\xrightarrow {R_n \rightarrow R_n + R_1}}
\hspace{0.2cm} P_1(\lambda) \tiny{\xrightarrow {R_n \rightarrow R_n + R_2}}
\hspace{0.2cm} P_2(\lambda) \hspace{0.2cm} \tiny{\xrightarrow {R_n \rightarrow R_n + R_3}} 
\cdots \hspace{0.2cm} \tiny{\xrightarrow{R_n \rightarrow R_n + R_{n-1}}} \hspace{0.2cm}
\widetilde{P}(\lambda)$.
\end{center}
At the $(n-1)^{th}$ step we get the matrix polynomial 
\begin{center}
$\widetilde{P}(\lambda) = \begin{bmatrix*}[c]
\displaystyle \sum_{k=0}^{m}a_{11}^{(k)}\lambda^k & \cdots & \displaystyle 
\sum_{k=0}^{m}a_{1(n-1)}^{(k)}\lambda^k & \displaystyle \sum_{k=0}^{m}\Big(1 
-\sum_{j=1}^{n-1}a^{(k)}_{1j}\Big) \lambda^k \\
\vdots & \ddots & \vdots & \vdots \\
\displaystyle \sum_{k=0}^{m}a_{(n-1)1}^{(k)}\lambda^k & \cdots & \displaystyle 
\sum_{k=0}^{m}a_{(n-1)(n-1)}^{(k)}\lambda^k & \displaystyle \sum_{k=0}^{m}\Big(1 - 
\sum_{j=1}^{n-1}a^{(k)}_{(n-1)j}\Big) \lambda^k \\
\displaystyle \sum_{k=0}^{m}\lambda^k & \cdots & \displaystyle \sum_{k=0}^{m}\lambda^k & 
\displaystyle \sum_{k=0}^{m}\lambda^k
\end{bmatrix*}$.
\end{center}
Note that the determinant of $P(\lambda)$ is same as the determinant of 
$\widetilde{P}(\lambda)$ 
(see section S1.1, \cite{GLR-Book-2}). That is, $\text{det}P(\lambda)= 
\text{det}\widetilde{P}(\lambda) = (\lambda^m + \dots + \lambda +1)q(\lambda)$. The 
conclusion follows as the roots of $\lambda^m + \dots + \lambda +1 $ are eigenvalues of $P(\lambda)$.        
\end{proof}

Let us remark that Corollary \ref{Cor-3-doubly-stochastic} follows from Theorem 
\ref{Thm-4-doubly-stochastic}. A proof of Corollary \ref{Cor-3-doubly-stochastic} has 
been added to illustrate Theorem \ref{Thm-1-eigenvalue-doubly stochastic-coefficients}.
  
\subsection{Location of eigenvalues of matrix polynomials with Schur stable matrix 
coefficients.}\label{sec-2.3}\hspace*{\fill} 

\medskip
We now move on to matrix polynomials whose coefficients are Schur stable matrices. Recall 
that a square matrix $A$ is said to be Schur stable if all its eigenvalues are of modulus
less than $1$. In this section we determine the eigenvalue location of matrix 
polynomials whose coefficients are Schur stable matrices; more generally we consider the 
matrix coefficients whose eigenvalues are of modulus less than $r$, for some $r > 0$. Let 
$M_r$ be the collection of square matrices whose eigenvalues are of modulus less than 
$r$, for some $r > 0$. Let $\mathcal{S}_r = \big\{P(\lambda)= I \lambda^m + 
A_{m-1}\lambda^{m-1} + \cdots + A_0: A_i \in M_r$ are commuting $n\times n$ matrices and 
$m,n \in \mathbb{N}\big\}$, and let $\sigma(\mathcal{S}_r)=\big\{|\lambda_0|: \lambda_0$ 
is an eigenvalue of $P(\lambda) \in \mathcal{S}_r \big\}$. Our first result is the following.

\begin{theorem}\label{Thm-1-eigenvalue-Schur stable-coefficients}
Let $P(\lambda)= I\lambda^m + A_{m-1}\lambda^{m-1} + \cdots + A_0 \in \mathcal{S}_r$. 
If $\lambda_0 \in \mathbb{C}$ is an eigenvalue of $P(\lambda)$, then $|\lambda_0| 
< r+1$.
\end{theorem}

\begin{proof}
Consider the corresponding block companion matrix of $P(\lambda)$ given by
$C=\begin{bmatrix}
0 & I & 0 & \cdots & 0\\
0 & 0 & I & \cdots & 0 \\
\vdots & \vdots & \vdots & \ddots & \vdots \\
0 & 0 & 0 & \cdots & I\\
-A_0 & -A_1 & -A_2 & \cdots & -A_{m-1}
\end{bmatrix}$. Since $A_i$'s commute with each other, there exists a unitary matrix 
$U$ such that $ U^*A_iU=T_i$, for all $i = 0, 1, \cdots, m-1$, where, the matrices $T_i$'s 
are upper triangular. Let us write $T_i = \begin{bmatrix}
a_{11}^{(i)} & a_{12}^{(i)} & \cdots & a_{1n}^{(i)}\\
0 & a_{22}^{(i)} & \cdots &  a_{1n}^{(i)}\\
\vdots & \ddots & & \vdots\\
0 & 0 & \cdots &  a_{nn}^{(i)}
\end{bmatrix} $ for $ i = 0,1,\cdots, m-1$. Then $ a_{11}^{(i)}, a_{22}^{(i)}, 
\cdots, a_{nn}^{(i)}$ are the eigenvalues of $A_i$. Hence $\big|a_{kk}^{(i)}\big|
< r$, for $k = 1,2, \cdots, n$ and $i = 0,1,\dots,m-1$. Let 
$Q(\lambda):=U^*P(\lambda)U = I\lambda^m+T_{m-1} \lambda^{m-1} +\cdots + T_1\lambda+T_0$. 
Then we have $\text{det}(P(\lambda)) =\text{det}(Q(\lambda)) = 
\displaystyle \prod_{k=1}^{n}\big(\lambda^m+ a_{kk}^{(m-1)}\lambda^{m-1}+ \cdots + 
a_{kk}^{(0)}\big)$. Therefore, the eigenvalues of $P(\lambda)$ are the roots of the 
polynomials, $\lambda^m + a_{kk}^{(m-1)}\lambda^{m-1} + \cdots + a_{kk}^{(0)}$, for 
$1\leq k\leq n $. Hence, if $\lambda_0$ is an eigenvalue of $P(\lambda)$, then 
$\lambda_0$ is a root of the polynomial $\lambda^m + \lambda^{m-1}a_{kk}^{(m-1)} + 
\cdots + a_{kk}^{(0)}$ for some $k$, $1\leq k \leq n $. Then by Lemma \ref{Lem-Cauchy 
bound for roots of polynomials} we have, 
$|\lambda_0| \leq 1+\max \{|a_{kk}^{(m-1)}|,\cdots,|a_{kk}^{(0)}|\} < r+1$. 
Therefore, $|\lambda_0| < r+1$. This proves the theorem.
\end{proof}

The following corollary is immediate.

\begin{corollary}\label{Cor-1-eigenvalue-Schur stable-coefficients}
Let $P(\lambda)= I\lambda^m+A_{m-1}\lambda^{m-1}+\cdots+A_0$, where $A_i$'s  are 
commuting Schur stable matrices. If $\lambda_0 \in \mathbb{C}$ is an eigenvalue of 
$P(\lambda)$, then $|\lambda_0|<2$.
\end{corollary}

A typical example where such a problem arises in applications is as follows. This is 
taken from \cite{Betcke-Higham-Tisseur}. Consider the following quadratic eigenvalue 
problem arising from a linearly damped mass-spring system:
$$P(\lambda)v=(M\lambda^2+C\lambda+K)v =0, $$
where $M = I, C = 10T, K=5T$, with $T$ being the tridiagonal matrix given by 
$T =\begin{bmatrix}
3 & -1 & &\\
-1 & \ddots & \ddots & \\
 & \ddots & \ddots & -1\\
 & & -1 & 3
\end{bmatrix}$. Note that $||T||_\infty = 5$, thus $||C||_\infty=50, 
||K||_\infty=25$. Therefore the eigenvalues of $M, C$ and $K$ lie inside the disc of 
radius $50 + \epsilon$ for any $\epsilon > 0$. Hence by Theorem 
\ref{Thm-1-eigenvalue-Schur stable-coefficients}, if $\lambda_0$ is an eigenvalue of 
$P(\lambda)$, then $|\lambda_0| < 50 + \epsilon + 1$. Since $\epsilon > 0$ is arbitrary, 
we infer that $|\lambda_0| \leq 51$.
These have also been verified numerically for sufficiently large size matrices. We now 
prove that the bounds obtained in Theorem \ref{Thm-1-eigenvalue-Schur stable-coefficients} 
are optimal.

\begin{theorem}\label{Thm-2-eigenvalue-Schur stable-coefficients}
$\sup (\sigma(\mathcal{S}_r))=r+1$ and $\inf (\sigma(\mathcal{S}_r))= 0 $.
\end{theorem}

\begin{proof}
Let $P(\lambda)=I\lambda +A_0 \in \mathcal{S}_r$, where $A_0$ is an $n\times n$ 
zero matrix. Then $0$ is an eigenvalue of $P(\lambda)$. Hence 
$0 \in \sigma(\mathcal{S}_r)$. Therefore $\inf (\sigma(\mathcal{S}_r))=0$. We now 
prove that $\sup (\sigma(\mathcal{S}_r))=r+1$. By 
Theorem \ref{Thm-1-eigenvalue-Schur stable-coefficients}, $r+1$ is an 
upper bound for $\sigma(\mathcal{S}_r)$. Let $1 < R < r+1$. Suppose $R$ is an upper bound 
for $\sigma(\mathcal{S}_r)$. For $m \in \mathbb{N}$, define $w_m(\lambda):=\lambda^m -
(r-\frac{1}{n})\lambda^{m-1}-\cdots - (r-\frac{1}{n})\lambda - (r-\frac{1}{n})$, where 
$n\in \mathbb{N}$ is such that $r > \frac{1}{n}$. Then  
$R^m \geq (r-\frac{1}{n})(R^{m-1}+\cdots+R+1)$. Otherwise, 
$R^m<(r-\frac{1}{n})(R^{m-1}+\cdots+R+1)$, from which we see that $w_m(R)= 
R^m-(r-\frac{1}{n})(R^{m-1}+\cdots+R+1)<0$. We also have, $(r+1)^m > r \big
((r+1)^{m-1}+\cdots+(r+1)+1\big) > (r-\frac{1}{n})\big((r+1)^{m-1}+\cdots+(r+1)+1\big)$. 
Therefore $w_m(r+1)>0$. The intermediate value theorem then guarantees that $w_m(\lambda)$ 
has a root in the interval $(R,r+1)$, say, $\lambda_0$. Now consider, 
$P(\lambda)=I \lambda^m +A\lambda^{m-1}+\cdots +A\lambda +A$, where 
$A= \begin{bmatrix}
-(r-\frac{1}{n}) & 0 \\
0 & -(r-\frac{1}{n})
\end{bmatrix} \in M_r$. Then $\text{det}(P(\lambda)) = \Big(\lambda^m -
(r-\frac{1}{n})\lambda^{m-1}-\cdots - (r-\frac{1}{n})\lambda - (r-\frac{1}{n})\Big)^2 
= (w_m(\lambda))^2$. Thus $\lambda_0 \in (R,r+1)$ is an eigenvalue of $P(\lambda)$. 
This implies that $\lambda_0 \in \sigma(\mathcal{S}_r)$, which is a contradiction to 
the assumption that $R$ is an upper bound for $\sigma(\mathcal{S}_r)$. Therefore we 
have, $R^m\geq(r-\frac{1}{n})(R^{m-1}+\cdots+R+1) = (r-\frac{1}{n})\Big(\dfrac{1-R^m}
{1-R}\Big)$. Since $1-R < 0$, we have $(1-R)R^m\leq \Big(r-\dfrac{1}{n} \Big)(1-R^m)$. 
This implies $r+1-R-\dfrac{1}{n}\leq \Big(r -\dfrac{1}{n}\Big)\dfrac{1}{R^m}$ for all  
$m,n \in \mathbb{N}$ with $r > \frac{1}{n}$. Letting $n,m \rightarrow \infty$ we have, 
$r+1-R\leq 0$; that is $r+1 \leq R$, a contradiction to the choice of $R$. 
Thus, $\sup (\sigma(\mathcal{S}_r))=r+1$.
\end{proof}

\medskip

\begin{remark}\label{Rem-Schur stable}
If we remove the commutativity condition on the coefficients of matrix polynomials in 
$\mathcal{S}_r$, then $\sigma(\mathcal{S}_r)$ is unbounded. To see this consider, 
$P_n(\lambda)=I\lambda^2+\begin{bmatrix*}[r]
0 & 0\\
-n & 0
\end{bmatrix*}\lambda + \begin{bmatrix*}[r]
0 & -n \\
0 & 0
\end{bmatrix*}$ for $n \in \mathbb{N}$. The coefficient matrices are non-commuting 
with $0$ as the only eigenvalue. Hence $P_n(\lambda) \in \mathcal{S}_r$ for all 
$n\in \mathbb{N}$ and any $r>0$. The block companion matrix $C_n$ of 
$P_n(\lambda)$ is given by, $C_n =\begin{bmatrix}
0 & 0 & 1 & 0\\
0 & 0 & 0 & 1 \\
0 & n & 0 & 0 \\
0 & 0 & n & 0
\end{bmatrix}$ whose characteristic polynomial is $\lambda(\lambda^3-n^2)$. Thus 
the eigenvalues of $C_n$ and hence of $P_n(\lambda)$ are $0, n^{\frac{2}{3}}, 
\Big(\dfrac{-1-i\sqrt{3}}{2}\Big)n^{\frac{2}{3}}$ and 
$\Big(\dfrac{-1+i\sqrt{3}}{2}\Big)n^{\frac{2}{3}}$.
Since $n^{\frac{2}{3}} \in \sigma(\mathcal{S}_r)$, for all $n \in \mathbb{N}$, 
$\sigma(\mathcal{S}_r)$ is unbounded.
\end{remark}

\medskip
\noindent
{\bf Acknowledgements:} Pallavi .B and Shrinath Hadimani acknowledge the Council of 
Scientific and Industrial Research (CSIR) and the University Grants Commission (UGC), 
Government of India, for financial support through research fellowships.

\bibliographystyle{amsplain}

\begin{thebibliography}{99}

\bibitem{Betcke-Higham-Tisseur}
T. Betcke, N. J. Higham, V. Mehrmann, C. Schr\"{o}der and F. Tisseur, \emph{N{LEVP}: a 
collection of nonlinear eigenvalue problems}, Association for Computing Machinery. 
Transactions on Mathematical Software, 39(2) (2013), Art. 7, 28.

\bibitem{Bini-Noferini-Sharify}
D. A. Bini, V. Noferini and M. Sharify, \emph{Locating the eigenvalues of matrix 
polynomials}, SIAM Journal on Matrix Analysis and Applications, 34(4) (2013), 1708-1727.

\bibitem{Cameron}
T. R. Cameron, \emph{Spectral bounds for matrix polynomials with unitary coefficients}, 
Electronic Journal of Linear Algebra, 30 (2015), 585-591.

\bibitem{Cauchy}
A. L. Cauchy, \emph{Exercises de mathematique}, Oeuvres, 9 (1829), 122.

\bibitem{GLR-Book-2}
I. Gohberg, P. Lancaster and L. Rodman, \emph{Matrix Polynomials}, $2^{nd}$ Edition, 
SIAM, 2009.
   	
\bibitem{Higham-Tisseur}
N. J. Higham and F. Tisseur, \emph{Bounds for eigenvalues of matrix polynomials}, 
Linear Algebra and its Applications, 358 (2003), 5-22.
   	
\bibitem{Le-Du-Nguyen}
C. T. Le, T. H. B. Du and T. D. Nguyen, \emph{On the location of eigenvalues of matrix polynomials}, 
Operators and Matrices, 13(4) (2019), 937-954.    	

\bibitem{PB-SH-SJ-1}
Pallavi .B, Shrinath Hadimani and Sachindranath Jayaraman, \emph{Hoffman-Wielandt type 
inequality for block companion matrices of certain matrix polynomials}, arXiv:2208.10149.

\bibitem{Tisseur-Higham}
F. Tisseur and N. J. Higham, \emph{Structured pseudospectra for polynomial eigenvalue 
problems, with applications}, SIAM Journal on Matrix Analysis and Applications, 23(1) 
(2001), 187-208.

\bibitem{Tisseur-Meerbergen}
F. Tisseur and K. Meerbergen, \emph{The quadratic eigenvalue problem}, SIAM Review, 43(2) 
(2021), 235-286.

\end{thebibliography}

\end{document}